\documentclass[12pt,a4paper]{article}

\usepackage{a4wide}

\usepackage[latin1]{inputenc}
\usepackage[T1]{fontenc}
\usepackage[english]{babel}

\usepackage{amsmath}
\usepackage{amssymb}
\usepackage{amsfonts,makeidx,amsthm,mathrsfs}

\newcommand{\F}{\mathscr{F}}

\newcommand{\Lr}{\mathscr{L}}
\newcommand{\dvol}{\frac{\omega^n}{n!}}
\newcommand{\dvollin}{\frac{\omegalin^n}{n!}}

\newcommand{\dsubvol}{\frac{\omega^{n-1}}{(n-1)!}}
\newcommand{\dsubsubvol}{\frac{\omega^{n-2}}{(n-2)!}}

\newcommand{\tr}{\mathrm{tr}}
\newcommand{\omegalin}{\omega_{\textrm{lin}}}

\newcommand{\Walg}{\mathbb{W}}
\newcommand{\Sp}{\mathrm{Sp}}

\newcommand{\g}{\mathfrak{g}}

\newcommand{\R}{\mathbb{R}}
\newcommand{\C}{\mathbb{C}}
\newcommand{\N}{\mathbb{N}}

\def\cl{\mathop{\mathrm{Cl}}\nolimits}

\def\ham{\mathop{\mathrm{Ham}}\nolimits}

\def\diff{\mathop{\mathrm{Diff}}\nolimits}
\def\VM{\mathop{\mathfrak{X}(M)}\nolimits}

\newcommand{\Fl}{\mathrm{Fl}}

\newcommand{\Z}{\mathcal{Z}}
\newcommand{\Y}{\mathcal{Y}}
\newcommand{\A}{\mathcal{A}}
\newcommand{\V}{\mathcal{V}}
\newcommand{\T}{\mathcal{T}}
\newcommand{\E}{\mathcal{E}}
\newcommand{\J}{\mathcal{J}}

\newcommand{\D}{\mathcal{D}}
\newcommand{\RR}{\mathcal{R}}
\newcommand{\Rr}{\mathrm{R}}

\newcommand{\extwedge}{\stackrel{\circ}{\wedge}}
\newcommand{\WW}{\mathbb{W}}
\newcommand{\ddt}{\frac{d}{dt}}
\newcommand{\dds}{\frac{d}{ds}}
\newcommand{\ddto}{\left.\ddt\right|_{0}}
\newcommand{\ddso}{\left.\dds\right|_{0}}
\newcommand{\Om}{\Omega^{\diff_0}}
\newcommand{\Omtilde}{\widetilde{\Omega}^{\diff_0}}

\newcommand{\invnu}{\frac{1}{\nu}}

\newcommand{\W}{\mathcal{W}}

\newtheorem{theorem}{Theorem}[section]

\newtheorem{lemma}[theorem]{Lemma}
\newtheorem{cor}[theorem]{Corollary}
\newtheorem{prop}[theorem]{Proposition}
\newtheorem{defi}[theorem]{Definition}

\theoremstyle{definition} 
\newtheorem{ex}[theorem]{Exemple}

\theoremstyle{remark}

\newtheorem{rem}[theorem]{Remark}

\begin{document}

\renewcommand{\refname}{Bibliography}

\title{A formal moment map on $\textrm{Diff}_0(M)$}

\author{Laurent La Fuente-Gravy\\
	\scriptsize{laulafuent@gmail.com}\\
	\footnotesize{Universit\'e libre de Bruxelles and Haute-\'Ecole Bruxelles-Brabant}\\[-7pt]
	\footnotesize{Belgium} \\[-7pt]}

\maketitle

\begin{abstract}
In the framework of deformation quantization, we obtain a deformation of Donaldson moment map on $\textrm{Diff}_0(M)$, the connected component of the group of diffeomorphisms of a symplectic manifold $(M,\omega)$ admitting another symplectic structure $\chi$.
\end{abstract}

\noindent {\footnotesize {\bf Keywords:} Symplectic connections, Moment map, Deformation quantization, Hamiltonian diffeomorphisms, diffeomorphisms group.}

\tableofcontents


\section{Introduction}

We exhibit another formal moment map on an infinite dimensional space using the technique from \cite{LLFlast}. This formal moment map appears to be a deformation of Donaldson moment map \cite{Don1} on $\diff_0(M)$ the connected component of the group of diffeomorphism of a closed symplectic manifold $(M,\omega)$ in the presence of an extra symplectic form $\chi$. 

When the manifold $M$ is K\"ahler and both $\omega$ and $\chi$ are K\"ahler forms, Donaldson \cite{Don1} deduced from his moment map a flow of K\"ahler metrics named the J-flow. As pointed out by Chen \cite{Ch}, existence of solutions of the J-flow leads to informations on the Mabuchi K-energy which plays a key role in the cscK metric problem.

Our deformation framework is deformation quantization \cite{BFFLS}. We consider a Fedosov star product algebra bundle $\mathcal{V}$ over $\diff_0(M)$. On the fiber above $f\in \diff_0(M)$, the star product is obtained from Fedosov construction parametrised with the formal $2$-form $\nu f^*\chi$ and symplectic connection $\nabla^f$ depending on $f$, see Subsection \ref{sect:nablaf}.

We use the construction of Andersen--Masulli--Sch\"atz \cite{AMS} to obtain a formal connection on $\mathcal{V}$. We show its curvature acts by inner derivations, that is the $*$-commutator with a formal function.

Acting as in the prequantisation picture, we consider the trace of the curvature of the formal connection as a first Chern form of a line bundle. We show this form is a deformation of the symplectic form on $\diff_0(M)$ that is preserved by the action of $\ham(M,\omega)$, the group of Hamiltonian diffeomorphisms of $(M,\omega)$, on $\diff_0(M)$.

Finally, we show that a formal moment map for the action of $\ham(M,\omega)$ on $\diff_0(M)$ is given by the trace of star product. At first order, our result recovers Donaldson moment map.

The result of this paper explain what was presented as a co\"incidence in \cite{LLF} linking first terms of traces for star products and moment maps on infinite dimensional space. It justifies the closed Fedosov star product problem as a problem of finding zeroes of a formal moment map as suggested from the works \cite{LLF,LLF2} and \cite{FutLLF1,FutLLF2}.


\section{A bundle of Fedosov $*$-product algebras on $\diff_0(M)$}

Throughout this short paper, we consider a closed symplectic manifold $(M,\omega)$ of dimension $2m$ admitting another symplectic form $\chi$. We also deal with infinite dimensional manifolds and Lie groups, we will follow the theory from \cite{KHN}.

\subsection{The symplectic structure on $\diff_0(M)$} \label{subsect:symplstruct}

We consider the connected component $\diff_0(M)$ of the group of diffeomorphisms on $M$. It is a Fr\'echet manifold modeled on the space $\VM$ of smooth vector fields on $M$. At any point $f\in \diff_0(M)$, the tangent space $T_f \diff_0(M)$ is identified to the space $\Gamma(f^*TM)$ of smooth sections of the pullback bundle $f^*TM$. So that a tangent vector at $f$ is a \emph{vector field along} $f$.

Given $Y\in \VM$, we extend it as the right invariant vector field $\Y: f\mapsto \Y_f:= Y\circ f$  on $\diff_0(M)$. The vector field $\Y$ as a flow $\Fl_t^{\Y}$ obtained from the flow of $\phi_t^Y$ of $Y$ on $M$ and defined by
\begin{equation}
\Fl_t^{\Y}(f):=\phi_t^Y \circ f, \ \textrm{ for all } f\in \diff_0(M).
\end{equation}
This flow satisfies
\begin{equation*}
\ddt \Fl_t^{\Y}(f):=\Y_{\Fl_t^{\Y}(f)}, \ \textrm{ for all } f\in \diff_0(M).
\end{equation*}
The Lie bracket of two right invariant vector fields $\Y$ and $\Z$, build out of $Y,Z \in \VM$ is given by
\begin{equation}\label{eq:Liebra}
\left[\Y,\Z\right]^{\diff_0}_f :=[Y,Z]\circ f , \ \textrm{ for all } f\in \diff_0(M),
\end{equation}
where the bracket on the RHS is the usual bracket of vector fields on $M$.

The symplectic form on $\diff_0(M)$ we will study is obtained from $\chi$ and $\omega$, it is defined as
\begin{equation*}
\Om_f(\Y_f,\Z_f):= \int_M \chi_{f(\cdot)}(Y_{f(\cdot)}, Z_{f(\cdot)}) \dvol.
\end{equation*}
To check $d^{\diff_0}\Om$ vanishes, one uses the fact that $d\chi=0$ and remember that the formula of the differential for forms on $\diff_0(M)$ writes for a $p$-form $\Theta$ on $\diff_0(M)$, 
\begin{eqnarray*}
(d^{\diff_0}\Theta)_f (Y_0\circ f,  \ldots, Y_p\circ f) & := &  \sum_{i=1}^p (-1)^i(\Y_i)_f\left(\Theta(\Y_0,\ldots \hat{\Y_i} \ldots, \Y_p)\right) \\
 & \hspace{-3cm} + & \hspace{-2cm}  \sum_{0\leq i<j \leq p} (-1)^{i+j+1} \Theta_f(\left[\Y_i,\Y_j\right]^{\diff_0(M)}, \Y_0,\ldots, \hat{\Y_i},\ldots,\hat{\Y_j},\ldots,\Y_p),
\end{eqnarray*}
for $f\in \diff_0(M)$ and $\Y_0,  \ldots,\Y_p$ right invariant vector fields obtained from $Y_0,\ldots Y_p$ vector fields on $M$. 
Note that we decorate Lie bracket and differential on $\diff_0(M)$ with the superscript $\diff_0$ to prevent from confusion with the corresponding notions on $M$.

We consider the group $\ham(M,\omega)$ of Hamiltonian diffeomorphisms of $(M,\omega)$. It acts on the right on $\diff_0(M)$ by
\begin{equation*}
\varphi\cdot f := f\circ \varphi, \textrm{ for } f\in \diff_0(M) \textrm{ and } \varphi \in \ham(M,\omega).
\end{equation*}
One checks this action preserves the symplectic form $\Om$.
The infinitesimal action of a Hamiltonian vector field $X_H$ (i.e. $X_H$ is defined by $\imath(X_H)\omega=dH$ for $H\in C^{\infty}(M)$) is then given by
\begin{equation} \label{eq:infXH}
X_H \cdot f := f_* X_H, \textrm{ for } f\in \diff_0(M),
\end{equation}
which is indeed a section of $f^*TM$.

Donaldson \cite{Don1} obtained a moment map for the above action of $\ham(M,\omega)$ on the symplectic manifold $(\diff_0(M),\Om)$. We will recover it from our construction of a formal moment map.

\subsection{Symplectic connections and $\diff_0(M)$} \label{sect:nablaf}

In the sequel, we will need to attach a Fedosov star product to a diffeomorphism $f\in \diff_0(M)$. Since one of the ingredients for constructing a Fedosov star product is a symplectic connection, we need a map that sends a diffeomorphism to a symplectic connection.

To obtain such a map, let us recall how the existence of a symplectic connection on any symplectic manifold is settled. First, a symplectic connection on $(M,\omega)$ is a connection on $TM$ that leaves $\omega$ parallel and has no torsion. Starting from any torsion-free connection $\nabla$ on $M$ (which always exists), one turns it into a symplectic connection $\widetilde{\nabla}$. One first define a tensor $N^{\nabla}$ by
\begin{equation*} \label{eq:tensorN}
\omega(N^{\nabla}(X,Y),Z):= (\nabla_X\omega)(Y,Z) \textrm{ for all } X,Y,Z \in TM.
\end{equation*}
Then, one checks by direct computation that the connection
\begin{equation*}
\widetilde{\nabla}_X Y:= \nabla_X Y + \frac{1}{3}N^{\nabla}(X,Y) + \frac{1}{3}N^{\nabla}(Y,X),
\end{equation*}
for $X,Y \in \VM$, defines a symplectic connection. 

Any two torsion-free connections $\nabla$ and $\widetilde{\nabla}$ differ from each other by a tensor $A(\cdot)\cdot \in \Gamma(TM\otimes S^2T^*M)$. Recall that if $\nabla$ is symplectic, then $\widetilde{\nabla}$ is symplectic if and only if $\underline{A}(\cdot,\cdot,\cdot):=\omega(\cdot,A(\cdot)\cdot)$ is a completely symmetric $3$-tensor on $M$. Also, recall that the space of symplectic connections $\E(M,\omega)$ is an affine space modelled on the space $\Gamma(S^3T^*M)$ of completely symmetric $3$-tensors on $M$.

Now, consider a symplectic connection $\nabla$ and a diffeomorphism $f$. One builds the torsion-free connection $f^*\nabla$ defined for $X,Y \in \VM$ by
\begin{equation*}
(f^*\nabla)_X Y:=f^{-1}_*(\nabla_{f_* X} f_*Y).
\end{equation*}

\begin{defi} \label{def:nablaf}
Choose $\nabla \in \E(M,\omega)$ a symplectic connection, define \emph{the map $\pi_{\nabla}$}
\begin{equation*}
\pi_{\nabla}:\diff_0(M) \rightarrow \E(M,\omega):f \mapsto \pi_{\nabla}(f):=\nabla^f,
\end{equation*}
where $(\nabla^f)_X Y:=(f^*\nabla)_X Y+ \frac{1}{3} N^{f^*\nabla}(X,Y) +\frac{1}{3}N^{f^*\nabla}(Y,X)$ and $X,Y\in \VM$.
\end{defi}

\begin{rem}
Similarly, one may also define a map from the space of torsion-free connections with values in the space of symplectic connections.
\end{rem}

Our construction of equivariant formal moment map will rely on the next proposition showing the equivariance of $\pi_{\nabla}$.

\begin{prop} \label{prop:nablaphi}
Let $\nabla$ be a symplectic connection, $f\in \diff_0(M)$ and write $f^*\nabla:=\nabla+A^f$ for $A^f\in \Gamma(TM\otimes S^2T^*M)$. \\
Then, writing $\pi_{\nabla}(f)= \nabla + B^f$, we have $\underline{B^f}(X,Y,Z)=\frac{1}{3} \stackrel{\curvearrowright}{\underset{XYZ}{\oplus}}\omega(X,A^f(Y)Z)$ for $X,Y,Z \in TM$ and $\stackrel{\curvearrowright}{\oplus}$ denotes the cyclic sum. \\
Moreover, $\pi_{\nabla}(f\circ \varphi)=\varphi^*\pi_{\nabla}(f)$ for all $\varphi \in \ham(M,\omega)$.
\end{prop}

\begin{proof}
To compute $\pi_{\nabla}(f)$, we first determine $N^{f^*\nabla}$ : for $X,Y,Z \in TM$,
\begin{eqnarray*}
\omega(N^{f^*\nabla}(X,Y),Z) & = & \left( (f^*\nabla)_X\omega\right)(Y,Z) \\
 & = & \omega(Z,A^f(X)Y)-\omega(Y,A^f(X)Z) \ \ \ \textrm{because } \nabla \textrm{ is symplectic.}
\end{eqnarray*}
Now, $B^f(X)Y = A^f(X)Y+\frac{1}{3} N^{f^*\nabla}(X,Y) +\frac{1}{3}N^{f^*\nabla}(Y,X)$ and one checks by direct computation that 
$$\omega(X,B^f(Y)Z)=\frac{1}{3} \stackrel{\curvearrowright}{\underset{XYZ}{\oplus}}\omega(X,A^f(Y)Z).$$

Now, for $\varphi\in\ham(M,\omega)$, we have $\varphi^* f^*\nabla=\varphi^*\nabla+\varphi^*A^f$. Hence, $A^{f\circ \varphi}=A^{\varphi}+\varphi^*A^f$. Remark that $\omega(\cdot,A^{\varphi}(\cdot)\cdot)$ is already completely symmetric as the action of $\ham(M,\omega)$ preserves the symplecticity of the connection $\nabla$. By applying the first part of this proposition and the fact that $\varphi$ preserves $\omega$, we have
\begin{eqnarray*}
\omega(X,B^{f\circ \varphi}(Y)Z) & = & \omega(X,A^{\varphi}(Y)Z)+\frac{1}{3} \stackrel{\curvearrowright}{\underset{XYZ}{\oplus}}\omega(\varphi_*X,A^f(\varphi_*Y)\varphi_*Z). \\
 & = & \omega(X,A^{\varphi}(Y)Z)+\omega(X,(\varphi^*B^f)(Y)Z).
\end{eqnarray*}
It means $\pi_{\nabla}(f\circ \varphi)=\nabla+A^{\varphi} + \varphi^*B^f=\varphi^*(\nabla+B^f)=\varphi^*\pi_{\nabla}(f)$
\end{proof}

\subsection{Fedosov star products}

We summarise the construction of Fedosov star products \cite{fed2}.

On $(M,\omega)$, consider a basis $\{e_1,\ldots,e_{2n}\}$ a basis of $T_xM$ at $x\in M$ for which $\omega_{ij}:=\omega(e_i,e_j)$. We consider the dual basis $\{y^1,\ldots, y^{2n}\}$ of $T^*_xM$. The formal Weyl algebra $\WW_x$ at $x\in M$ is the algebra of formal symmetric forms on $T_xM$ of the form:
$$a(y,\nu):=\sum_{2k+r=0}^{\infty} \nu^r a_{r,i_1\ldots i_k}y^{i_1}\ldots y^{i_k}$$
where $a_{r,i_1\ldots i_k}$ symmetric in $i_1\ldots i_k$ and $2k+r$ is the total degree, with product $\circ$ given, for $a(y,\nu)$ and  $b(y,\nu)\in \WW_x$ by
\begin{eqnarray*}
(a\circ b)(y,\nu) & := &\left.\left( \exp\left(\frac{\nu}{2}\Lambda^{ij} \partial_{y^i} \partial_{z^j}\right)a(y,\nu)b(z,\nu)\right)\right|_{y=z} \nonumber \\
\end{eqnarray*}

We globalise the above over $M$ to get the formal Weyl algebra bundle $\W:=\bigsqcup_{x\in M}\WW_x$. We also consider differential forms with values in the $\W$ by $\W\otimes \Lambda M$ whose sections are tensors on $M$ that write locally as:
\begin{equation*} \label{eq:sectionofW}
\sum_{2k+l\geq 0,\, k,l\geq 0, p\geq 0} \nu^k a_{k,i_1\ldots i_l,j_1\ldots j_p}(x)y^{i_1}\ldots y^{i_l}dx^{j_1}\wedge \ldots \wedge dx^{j_p}.
\end{equation*}
The $a_{k,i_1\ldots i_l,j_1\ldots j_p}(x)$ are symmetric in the $i$'s and antisymmetric in the $j$'s. The space $\Gamma \W\otimes \Lambda^*M$ is filtered with respect to the total degree
$$\Gamma \W\otimes \Lambda^*M \supset \Gamma \W^1\otimes \Lambda^*M\supset \Gamma \W^2\otimes \Lambda^*M \supset \ldots.$$
Extending fiberwisely the $\circ$-product turns $\Gamma \W\otimes \Lambda^*M$ into an algebra. That is, for $a, b\in \Gamma \W$ and $\alpha, \beta\in \Omega^*(M)$, we define
$(a\otimes \alpha) \circ (b\otimes \beta) := a\circ b \otimes \alpha\wedge \beta$. The graded commutator $[s,s']:=s\circ s'- (-1)^{q_1q_2}s'\circ s$ where $s$, resp. $s'$ are of anti-symmetric degree $q_1$, resp. $q_2$ makes $\W$-valued forms a graded Lie algebra.

A symplectic connection $\nabla$ on $(M,\omega)$ induces a derivation $\partial$ of degree $+1$ (anti-symmetric degree) on $\W$-valued forms by :
\begin{equation*}
\partial a := da + \frac{1}{\nu}[\overline{\Gamma},a]  \textrm{ for } a\in \Gamma \W\otimes \Lambda M,
\end{equation*}
where $\overline{\Gamma}:=\frac{1}{2}\omega_{lk}\Gamma^k_{ij}y^ly^jdx^i$, for $\Gamma^k_{ij}$ the Christoffel symbols of $\nabla$ on a Darboux chart, making  $\omega_{lk}\Gamma^k_{ij}$ completely symmetric in $i,j,l$. 

Setting $\overline{R}:= \frac{1}{4} \omega_{ir}R^r_{jkl}y^iy^jdx^k\wedge dx^l$, for $R^r_{jkl}:=\left(R(\partial_k,\partial_l)\partial_j\right)^r$ the components of the curvature tensor of $\nabla$, the curvature of $\partial$ is
\begin{equation*}
\partial\circ \partial\, a := \frac{1}{\nu}[\overline{R},a].
\end{equation*}

To make this connection flat, we consider connections on $\Gamma \W$ of the form
\begin{equation*} \label{eq:defD}
D a:=\partial a - \delta a + \frac{1}{\nu}[r,a],
\end{equation*}
where $r$ is a $\W$-valued $1$-form and $\delta$ is defined by
\begin{equation*} \label{eq:deltadef}
\delta(a) := dx_k\wedge \partial_{y_k} a=-\frac{1}{\nu}[\omega_{ij}y^i dx^j,a].
\end{equation*}
The curvature of $D$ is
\begin{equation*}
D^2 a = \frac{1}{\nu}\left[\overline{R} + \partial r - \delta r + \frac{1}{2\nu}[r,r]-\omega,a\right].
\end{equation*}

Define 
$$\delta^{-1} a_{pq}:= \frac{1}{p+q}y^ki(\partial_{x^k})a_{pq} \textrm{ if } p+q>0 \textrm{ and } \delta^{-1}a_{00}=0,$$
where $a_{pq}$ is a $q$-forms with $p$ $y$'s and $p+q>0$. Fedosov showed \cite{fed2}, for any given closed central $2$-form $\Omega$,
there exists a unique solution $r \in \Gamma \W \otimes \Omega^1 M$ with $\W$-degree at least $3$ of equation:
\begin{equation*}\label{eq:req}
\overline{R} + \partial r - \delta r + \frac{1}{\nu}r\circ r = \Omega,
\end{equation*}
and satisfying $\delta^{-1}r=0$. Becaue $\Omega$ is central for the $\circ$-product, it makes $D$ flat.

Set $D$ the flat connection obtained as above. Flat sections $\Gamma \W_{D} := \{a\in \Gamma \W | D a=0\}$ form an algebra for the $\circ$-product since $D$ is a derivation. The symbol map is defined by $\sigma :a\in \Gamma \W_{D} \mapsto \left.a\right|_{y=0}\in C^{\infty}(M)[[\nu]]$. Fedosov showed \cite{fed2} that $\sigma$ is a bijection with inverse $Q$ defined by 
\begin{equation*} \label{eq:defQ}
Q:=\sum_{k\geq 0} \left(\delta^{-1}(\partial + \frac{1}{\nu}[r,\cdot])\right)^k.
\end{equation*}
Hence, the $\circ$-product induces a star product $*$ on $C^{\infty}(M)[[\nu]]$, called Fedosov star product.

The following is a technical lemma we will need. Its proof can be found in \cite{fed4}.
\begin{lemma}\label{lemme:Dinverse}
Suppose $b \in \Gamma(W)\otimes \Lambda^1M$ satisfy $Db = 0$. Then the equation $Da = b$ admits a
unique solution $a \in \Gamma(W)$, such that $a|_{y=0} = 0$, it is given by
$$b=D^{-1}a:=-Q(\delta^{-1}a).$$
\end{lemma}

In the sequel, to emphasize the dependence of $*$ (resp. $r$, $D$ and $Q$) in the choices $\nabla$ and $\Omega$, we will write $*_{\nabla,\Omega}$ (resp. $r^{\nabla,\Omega}$, $D^{\nabla,\Omega}$ and $Q^{\nabla,\Omega}$) and simply $*_{\nabla}$ (resp. $r^{\nabla}$,$D^{\nabla}$ and $Q^{\nabla}$) when $\Omega=0$.

As in \cite{LLFlast}, our technique relies on a canonical lift of smooth path of symplectic connections and formal series of closed $2$-forms to isomorphisms of Fedosov star product algebra, which comes from Fedosov \cite{fed2}.

To state it, we need sections of the extended bundle $\W^+\supset \W$ which are locally of the form 
\begin{equation*} \label{eq:sectionofW+}
\sum_{2k+l\geq 0, l\geq 0} \nu^k a_{k,i_1\ldots i_l}(x)y^{i_1}\ldots y^{i_l}.
\end{equation*}
similar to \eqref{eq:sectionofW}, with $p=0$, but we allow $k$ to take negative values, the total degree $2k+l$ of any term must remain nonnegative and in each given nonnegative total degree there is a finite number of terms.

\begin{theorem}\label{theor:smoothisom}
Consider smooth paths $t\in [0,1] \mapsto \nabla^t\in \E(M,\omega)$ and $t\in [0,1] \mapsto \chi_t\in \Omega^2(M)$.
Assume that for all $t$ : $\frac{d}{dt}\chi_t = d\theta_t$ for some smooth path $\theta_t\in \Omega^1(M)$.\\
Then there exists maps $B_t:\Gamma\W \rightarrow \Gamma \W$ defined by
$$B_t a:=v_t\circ a \circ v_t^{-1}$$
for $v_t \in \Gamma \W^+$ being the unique solution of the initial value problem:
\begin{equation*} \label{eq:vt}
\left\{\begin{array}{rcl}
\frac{d}{dt}v_t & = & \frac{1}{\nu}h_t\circ v_t \\
v_0 &= & 1
\end{array}\right.
\end{equation*}
with
\begin{equation*} \label{eq:ht}
h_t:=-(D^{\nabla^t,\,\nu\chi_t})^{-1}\left(\ddt \overline{\Gamma}^{\nabla^t}+\ddt r^{\nabla^t,\,\nu\chi_t} - \nu \theta_t\right).
\end{equation*}
Moreover, $B_t(D^{\nabla^0,\,\nu\chi_0} a)= D^{\nabla^t,\,\nu\chi_t}(B_t a)$ for all $a\in \Gamma\W$ so that $$\left.B_t\right|_{\Gamma\W_{D^{\nabla^0,\,\nu\chi_0}}}:\Gamma\W_{D^{\nabla^0,\,\nu\chi_0}} \rightarrow \Gamma \W_{D^{\nabla^t,\,\nu\chi_t}}$$ is an isomorphism of flat sections algebras and hence
\begin{equation} \label{eq:equivparall}
\sigma\circ B_t \circ Q^{\nabla^0,\,\nu\chi_0}:(C^{\infty}(M)[[\nu]],*_{\nabla^0,\,\nu\chi_0}) \rightarrow (C^{\infty}(M)[[\nu]],*_{\nabla^t,\,\nu\chi_t})
\end{equation}
is an equivalence of star product algebras.
\end{theorem}

\noindent The proof can be found in \cite{fed2}.

Similar to \cite{LLFlast}, $h_t$ above depends polynomially on $\nabla^t$, $\chi_t$ and $\theta_t$ and their covariant derivatives making the path $t\mapsto h_t$ smooth. It implies $t\mapsto v_t$ is smooth as well.

\subsection{A *-product algebra bundle over $\diff_0(M)$}

Consider a fixed symplectic connection $\nabla$ on $M$. We consider the family of Fedosov star products $\{*_{\nabla^f, \,\nu f^*\chi}\}_{f\in \diff_0(M)}$.

\begin{defi}
To the family of Fedosov star products $\{*_{\nabla^f, \,\nu f^*\chi}\}_{f\in \diff_0(M)}$, we attach a star product algebra bundle $\mathcal{V}$ over $\diff_0(M)$ defined by 
$$\mathcal{V}:= \mathcal{V}(\nabla)=\diff_0(M) \times C^{\infty}(M)[[\nu]] \stackrel{p}{\rightarrow} \diff_0(M),$$ 
with fiber $p^{-1}(f)$ above $f\in \diff_0(M)$ endowed with the Fedosov star product $*_{\nabla^f, \,\nu f^*\chi}$. 
\end{defi}

The above bundle depends on the choice of a symplectic connection. Consider two symplectic connections $\nabla$ and $\widetilde{\nabla}$ and produce the two star product families $\{*_{\nabla^f, \,\nu f^*\chi}\}_{f\in \diff_0(M)}$ and $\{*_{\widetilde{\nabla}^f, \,\nu f^*\chi}\}_{f\in \diff_0(M)}$ as well as the two bundle $\V(\nabla)$ and $\V(\widetilde{\nabla})$. 

We show $\V(\nabla)$ and $\V(\widetilde{\nabla})$ are isomorphic as $*$-product algebra bundles. Indeed, we consider the segment $t\mapsto \nabla^t:=\nabla + t(\widetilde{\nabla}-\nabla)$ of symplectic connections. For all $f\in \diff_0(M)$, we apply Theorem \ref{theor:smoothisom} to the path $t\mapsto (\nabla^t)^f$, the constant path of $2$-forms $t\mapsto f^*\chi$ and the choice $\beta_t=0$, to obtain a $*$-product equivalence between the fibers of $\V(\nabla)$ and $\V(\widetilde{\nabla})$ above $f$.

We go on with the definition of a compatible formal connection $\D$ on sections of $\V$. Recall that being a \emph{formal connection} means that $\D$ acts on sections of $\V$ by $d^{\diff_0}+\beta$, with $\beta=\sum_{k=1} \nu^k \beta_k$ being a formal series of $1$-forms on $\diff_0(M)$ with values in differential operators on functions of $M$. The \emph{compatibility} with respect to the family of star products $\{*_{\nabla^f,\, \nu f^*\chi}\}_{f\in \diff_0(M)}$ is :
$$\D(F*_{\nabla^f,\, \nu f^*\chi}G)=\D(F)*_{\nabla^f,\, \nu f^*\chi} G + F*_{\nabla^f,\, \nu f^*\chi} \D(G),$$
for any sections $F,G$ of $\V$.

Consider a smooth path $t\mapsto f_t \in \diff_0(M)$ with tangent vector at $t$ given by $\frac{d}{dt}f_t=Y_t\circ f_t$ that is a vector field on $M$ along $f_t$, we will define $\D$ through a notion of a parallel lift of the path $f_t$ using Theorem \ref{theor:smoothisom}. For this, we use the paths $t\mapsto \nabla^{f_t}$ and $t \mapsto f_t^*\chi$. A natural candidate for $\theta_t$ needed in Theorem \ref{theor:smoothisom} is $\theta_t:=f_t^*\imath(Y_t)\chi$, which comes from the computation of $\frac{d}{dt} f_t^*\chi=df_t^*\imath(Y_t)\chi$. Hence, with these data we obtain equivalences of $*$-products $\sigma\circ B_t \circ Q^{\nabla^{f_0},\,\nu f_0^*\chi}$.

\begin{defi} \label{def:formalconn}
For $Y\circ f\in T_{f}\diff_0(M)$, with $\phi_t^Y$ the flow of the vector field $Y$ on $M$, set :
\begin{itemize}
\item the \emph{connection $1$-form} $\alpha\in \Omega^1(\E(M,\omega),\Gamma\W^3)$ by
\begin{equation*} \label{eq:defalpha}
\alpha_{f}(Y\circ f):=(D^{\nabla^f, \, \nu f^* \chi})^{-1}\left(\left.\frac{d}{dt}\right|_0\overline{\Gamma}^{\nabla^{\phi_t^Y\circ f}}+ \ddto r^{\nabla^{\phi_t^Y \circ f},\,\nu(\phi_t^Y\circ f)^*\chi} - \nu f^*\imath(Y)\chi\right),
\end{equation*}
for $\frac{d}{dt}\overline{\Gamma}^{\nabla^{\phi_t^Y\circ f}}=\frac{1}{2}\omega_{lk}\left(\frac{d}{dt}\Gamma^{\nabla^{\phi_t^Y\circ f}}\right)_{ij}^ky^ly^jdx^i$,
\item the $1$-form $\beta$ with values in formal differential operators:
\begin{equation*}
\beta_{f}(Y\circ f)(F):= \left.\invnu[\alpha_{f}(Y\circ f),Q^{\nabla^f,\,\nu f^*\chi}(F)]\right|_{y=0},\,\textrm{ for }F\in C^{\infty}(M)[[\nu]],
\end{equation*}
\item the \emph{formal connection} $\D:= d^{\diff_0}+\beta$.
\end{itemize}
\end{defi}

\begin{prop}
$\D$ is a formal connection on $\V$ compatible with the family of Fedosov star product $\{*_{\nabla^f,\, \nu f^*\chi}\}_{f\in \diff_0(M)}$.
Moreover, the parallel transport for $\D$ along the path $t\mapsto f_t\in \diff_0(M)$ is given by the equivalence of star product algebra obtained from Theorem \ref{theor:smoothisom} with paths $t\mapsto \nabla^{f_t}$, $t\mapsto f^*_t\chi$ and $\theta_t=f_t^*\imath(Y_t)\chi$, where $\frac{d}{dt}f_t=Y_t\circ f_t$.
\end{prop}

\begin{proof}
For the parallel transport property, we consider the equivalence induced by $v_t$ as in Theorem \ref{theor:smoothisom} with the data of the above statement. It means $v_t$ is generated by $h_t$ with
$$ h_t=-(D^{\nabla^{f_t},\,\nu f^*_t\chi})^{-1}\left(\ddt \overline{\Gamma}^{\nabla^{f_t}}+\ddt r^{\nabla^{f_t},\,\nu f^*_t\chi} - \nu f_t^*\imath(Y_t)\chi\right). $$
Starting from $F\in C^{\infty}(M)[[\nu]]$, seen as an element of the fiber of $\V$ above $f_0$, we propagate it above the path $f_t$ by
$$F(f_t):= \left. \left(v_t\circ Q^{\nabla^{f_0},\,\nu f_0^*\chi}(F)\circ v_t\right)\right|_{y=0}.$$
We have to check $(\D_{Y_t\circ f_t} F)(f_t)$ is zero. That is
\begin{eqnarray*}
 (\D_{Y_t\circ f_t} F)(f_t) & = & \frac{d}{dt} \left. \left(v_t\circ Q^{\nabla^{f_0},\,\nu f_0^*\chi}(F)\circ v_t\right)\right|_{y=0} + \beta_{f_t}(Y_t\circ f_t)(F(f_t)), \\
 & = & \frac{1}{\nu} \left.\left(\left[h_t, v_t\circ Q^{\nabla^{f_0},\,\nu f_0^*\chi}(F)\circ v_t\right]\right)\right|_{y=0}+\frac{1}{\nu}  \left.\left[\alpha_{f_t}(Y_t\circ f_t), v_t\circ Q^{\nabla^{f_0},\,\nu f_0^*\chi}(F)\circ v_t\right]\right|_{y=0},\\
& = & 0,
\end{eqnarray*}
which vanishes by definition of $\alpha$.

For the compatibility, we refer to the original computation in \cite{AMS}, see also \cite{LLFlast} for the corresponding statement.
\end{proof}

\section{A formal moment map picture on $\diff_0(M)$}

\subsection{The curvature of $\D$}

The curvature of $\D$ evaluated at $Y\circ f$ and $Z\circ f$ tangent vector at the point $f\in \diff_0(M)$ is defined by
$$\RR(Y\circ f,Z\circ f)F:=\D_{Y\circ f}(\D_{Z\circ f} F) - \D_{Z\circ f}(\D_{Y\circ f} F) - \D_{[Y,Z]\circ f} F,$$
for a section $F$ of $\V$. In the RHS above we use the natural extension of vector fields along $M$ as right invariant vector fields on $\diff_0(M)$, as well as the formula for the Lie bracket \eqref{eq:Liebra}.

\begin{theorem} \label{theor:Ralpha}
The curvature of $\D$ evaluated at the section $F$ of $\V$ and tangent vectors $Y\circ f$ and $Z\circ f$ at $f\in \diff_0(M)$ is given by
\begin{equation}\label{eq:Rdef}
\RR(Y\circ f,Z\circ f)F=\left.\invnu[\Rr(Y\circ f,Z\circ f),Q(F)]\right|_{y=0}
\end{equation}
for $\Rr(Y\circ f,Z\circ f)$ being the $2$-form with values in $\Gamma\W$ defined by
\begin{equation}\label{eq:Ralpha}
\Rr_f(Y\circ f,Z\circ f):=-\nu f^*(\chi(Y,Z)) +  d^{\diff_0}_{f}\alpha(Y\circ f,Z\circ f)+\invnu[\alpha_{f}(Y\circ f),\alpha_{f}(Z\circ f)],
\end{equation}
Moreover, 
\begin{itemize}
\item $\Rr_{f}(Y\circ f,Z\circ f)\in \Gamma \W_{D^{\nabla^f, \nu f^*\chi}}$,
\item $\left.\Rr_{f}(Y\circ f,Z\circ f)\right|_{y=0}=-\nu f^*(\chi(Y,Z))+ O(\nu^2)$.
\end{itemize}
\end{theorem}

\begin{proof}
In Equation \eqref{eq:Ralpha}, the terms involving the connection $1$-form $\alpha$ come from standard computation of $\RR$. The term $-\nu f^*(\chi(Y,Z))$ doesn't play any role in Equation \eqref{eq:Rdef}, it is just added to make sure $\Rr$ will take values in a space of flat sections.

To check  $\Rr_{f}(Y\circ f,Z\circ f)\in \Gamma \W_{D^{\nabla^f, \,\nu f^*\chi}}$, we compute $D^{\nabla^f, \,\nu f^*\chi}$ applied to the RHS of \eqref{eq:Ralpha}. First, because $f^*(\chi(Y,Z))$ is a function on $M$, 
$$D^{\nabla^f, \,\nu f^*\chi}  f^*(\chi(Y,Z))=d(f^*\chi(Y,Z)).$$
Now, we detail the terms of $D^{\nabla^f, \,\nu f^*\chi}d^{\diff_0}_{f}\alpha(Y\circ f,Z\circ f)$. To do that we use the flows $\phi^Y_t$, $\phi^Z_t$ and $\phi^{[Y,Z]}_t$ of $Y$, $Z$ and $[Y,Z]$ as vector fields on $M$ as well as the formula for the differential of forms which is still valid for forms with values in $\Gamma\W$.
We start with 
\begin{eqnarray}
D^{\nabla^f, \,\nu f^*\chi}\left(Y\circ f (\alpha_{\cdot}(Z\circ \cdot))\right) & = & \left.\frac{d}{dt}\right|_0 D^{\nabla^{\phi^Y_t\circ f},\, \nu (\phi_t^Y\circ f)^*\chi} \alpha_{\phi_t^Y\circ f}(Z\circ \phi_t^Y\circ f) \nonumber \\ 
 & & -  \left.\frac{d}{dt}\right|_0 D^{\nabla^{\phi^Y_t\circ f}, \,\nu (\phi_t^Y\circ f)^*\chi} \alpha_{f}(Z\circ f) \nonumber \\
 & = & \ddto \left(\ddso\left(\overline{\Gamma}^{\nabla^{\phi_s^Z\circ\phi_t^Y\circ f}}+ r^{\nabla^{\phi_s^Z \circ\phi_t^Y\circ f},\,\nu (\phi_s^Z \circ \phi_t^Y\circ f)^*\chi}\right) - \nu (\phi_t^Y \circ f)^*\imath(Z)\chi\right)\nonumber \\
 & & - \frac{1}{\nu}\left[\ddto\left(\overline{\Gamma}^{\nabla^{\phi_t^Y\circ f}}+r^{\nabla^{\phi_t^Y\circ f},\,\nu(\phi_t^Y\circ f)^*\chi}\right),\alpha_{f}(Z\circ f)\right]. \label{eq:Ddalpha1}
\end{eqnarray}
Similarly, we get
\begin{eqnarray}\label{eq:Ddalpha2}
D^{\nabla^f,\, \nu f^*\chi}\left(Z\circ f (\alpha_{\cdot}(Y\circ \cdot))\right) 
 & = & \ddto \left(\ddso\left(\overline{\Gamma}^{\nabla^{\phi_s^Y\circ\phi_t^Z\circ f}}+ r^{\nabla^{\phi_s^Y \circ\phi_t^Z\circ f},\,\nu(\phi_s^Y \circ \phi_t^Z\circ f)^*\chi}\right) - \nu (\phi_t^Z \circ f)^*\imath(Y)\chi\right)\nonumber \\
 & & - \frac{1}{\nu}\left[\ddto\left(\overline{\Gamma}^{\nabla^{\phi_t^Z\circ f}}+r^{\nabla^{\phi_t^Z\circ f},\,\nu(\phi_t^Z\circ f)^*\chi}\right),\alpha_{f}(Y\circ f)\right].
\end{eqnarray}
The last term from the differential of $\alpha$ gives
\begin{eqnarray}
D^{\nabla^f, \,\nu f^*\chi}\alpha_{\cdot}([Y,Z]\circ f)
 & = & \ddto \left(\overline{\Gamma}^{\nabla^{\phi^{[Y,Z]}_t\circ f}}+ r^{\nabla^{\phi^{[Y,Z]}_t\circ f},\,\nu(\phi^{[Y,Z]}_t\circ f)^*\chi}\right) - \nu  f^*\imath([Y,Z])\chi \label{eq:Ddalpha3} 
\end{eqnarray}
In the computation of $D^{\nabla^f, \,\nu f^*\chi}d^{\diff_0}_{f}\alpha(Y\circ f,Z\circ f)$, the terms in $\Gamma$ and $r$ from the first lines of \eqref{eq:Ddalpha1} and \eqref{eq:Ddalpha2} are compensated by the terms in $\Gamma$ and $r$ from \eqref{eq:Ddalpha3}.
The second line of \eqref{eq:Ddalpha1} minus the second line of \eqref{eq:Ddalpha2} are compensated by $$D^{\nabla^f, \,\nu f^*\chi}\invnu[\alpha_{f}(Y\circ f),\alpha_{f}(Z\circ f)],$$
where we use tha fact that the forms $\nu \imath(Y) \chi$ and $\nu \imath(Z) \chi$ are central. What remains is 
\begin{eqnarray*}
D^{\nabla^f, \nu f^*\chi}\Rr_{f}(Y\circ f,Z\circ f) & = & - d(f^*\chi(Y,Z)) - \ddto \nu (\phi_t^Y \circ f)^*\imath(Z)\chi +\ddto \nu (\phi_t^Z \circ f)^*\imath(Y)\chi \\
 & & +\, \nu  f^*\imath([Y,Z])\chi, \\
 & = & 0.
\end{eqnarray*}

\noindent Finally, because $\alpha(\cdot)$ is of degree at least 3, we have
$$\left.\Rr_{f}(Y\circ f,Z\circ f)\right|_{y=0}=-\nu f^*(\chi(Y,Z))+ O(\nu^2).$$
\end{proof}

\subsection{A formal symplectic form on $\diff_0(M)$}

Recall that a \emph{formal symplectic form} on a manifold $F$ is a formal power series of closed $2$-forms
$$\sigma:=\sigma_0+\nu\sigma_1+\ldots \in \Omega^2(F)[[\nu]],$$
which starts with a symplectic form $\sigma_0$. It is a \emph{formal deformation} of the symplectic form $\sigma_0$.

Acting as in the theory of finite dimensional vector bundles, we consider the trace of the curvature $\Rr$ of $\V$.

A \emph{trace} for a star product $\ast$ on a symplectic manifold $(M,\omega)$ is a map
$$ \tr : C_c^\infty(M)[[\nu]] \to \R[\nu^{-1},\nu]]$$
such that $ \tr([F,G]_\ast) = 0,$ for all $F,G\in C_c^\infty(M)[[\nu]]$.
A trace always exits for a given star product on $(M,\omega)$ and it is made unique by asking the following normalisation. Consider local equivalences $B$ of $*|_{C^\infty(U)[[\nu]]}$ with the Moyal star product $\ast_{\mathrm{Moyal}}$ on $U$ a contractible Darboux chart 
$B : (C^\infty(U)[[\nu]],\ast) \to (C^\infty(U)[[\nu]],\ast_{\mathrm{Moyal}})$
 so that $ BF\ast_{\mathrm{Moyal}} BG = B(F\ast G).$
We ask for the normalised trace to statisfy
\begin{equation*}\label{normalized_tr}
\tr(F) = \frac1{(2\pi\nu)^m} \int_M BF\ \frac{\omega^m}{m!}, \textrm{ for all } F\in C_c^\infty(U)[[\nu]].
\end{equation*}
The trace can be written as
$$\tr(F) =\frac1{(2\pi\nu)^m} \int_M F\rho\ \frac{\omega^m}{m!},$$
for $\rho\in C^{\infty}(M)[\nu^{-1},\nu]]$, called the trace density.

Let us denote by \emph{$\tr^{*_{\nabla^f,\,\nu f^*\chi}}$ the normalised trace} of the Fedosov star product $*_{\nabla^f,\,\nu f^*\chi}$ and by \emph{$\rho^{\nabla^f,\,\nu f^*\chi}$ its trace density}.

\begin{defi} \label{def:Omtilde}
Define the formal $2$-form $\Omtilde$ on $\diff_0(M)$ by
$$\Omtilde(Y\circ f, Z\circ f):=-(2\pi)^m\nu^{m-1} \tr^{*_{\nabla^f,\, \nu f^*\chi}}(\left.\Rr_f(Y\circ f, Z\circ f)\right|_{y=0}),$$
for $f\in \diff_0(M)$ and $Y\circ f, Z\circ f \in T_f\diff_0(M)$.
\end{defi}

\begin{theorem}\label{theor:omegatilde}
The formal $2$-form $\Omtilde$ is a formal symplectic form on $\diff_0(M)$ deforming $\Om$.\\
Moreover, the action of $\ham(M,\omega)$ on $\diff_0(M)$ preserves $\Omtilde$.
\end{theorem}

\begin{proof}
The proof is a standard computation similar to that in \cite{LLFlast}. First, one compute directly that $d^{\diff_0}\Omtilde=0$ on right invariant vector fields using the Lemma below which is a particular case of Theorem 3.1 from \cite{FutLLF2}.

\begin{lemma} \label{lemme:tracevariation}
Let $t\mapsto f_t$ be a smooth path in $\diff_0(M)$. Then
$$ \ddto \tr^{*_{\nabla^{f_t},\,\nu f_t^*\chi}}(F)=\tr^{*_{\nabla^{f_0},\,\nu f_0^*\chi}}\left(\left.\invnu[\alpha_{f_0}(\ddto f_t),Q^{\nabla^{f_0},\,\nu f_0^*\chi}(F)]\right|_{y=0}\right). $$
\end{lemma}

The fact that $\Omtilde$ is preserved by the action of $\ham(M,\omega)$, comes from the naturality of Fedosov construction. Indeed, the pull-back by $\varphi \in \ham(M,\omega)$ on $\Gamma \W \otimes \Lambda M$ maps Fedosov flat connections to Fedosov flat connections as :
$$ \varphi^*D^{\nabla^f,\,\nu f^*\chi}(\varphi^{-1})^*=D^{\nabla^{f\circ \varphi},\,\nu (f\circ \varphi)^*\chi},$$
where we use $\nabla^{f\circ \varphi}=\varphi^*\nabla^f$ from Proposition \ref{prop:nablaphi}.
restricts to an isomorphism of flat section algebras
$$\varphi^*: \Gamma \W_{\nabla^f,\,\nu f^*\chi} \rightarrow \Gamma \W_{\nabla^{f\circ \varphi},\,\nu (f\circ \varphi)^*\chi}, $$
where we use $\nabla^{f\circ \varphi}=\varphi^*\nabla^f$ from Proposition \ref{prop:nablaphi}. Hence, the normalised traces are related by 
$$\tr^{*_{\nabla^{f\circ \varphi},\,\nu (f\circ \varphi)^*\chi}}=\tr^{*_{\nabla^f,\,\nu f^*\chi}}\circ (\varphi^{-1})^*.$$
Also the connection $1$-forms are related by
$$\alpha_{f\circ \varphi}((\varphi\cdot)_* \mathcal{Y}) = \varphi^*\alpha_{f}(\mathcal{Y}),$$
for $\mathcal{Y}\in T_f\diff_0(M)$. From this, we deduce $$\tr^{*_{\nabla^{f\circ \varphi},\,\nu (f\circ \varphi)^*\chi}}(\left.\Rr_{f\circ  \varphi}((\varphi\cdot)_*\mathcal{Y},(\varphi\cdot)_* \mathcal{Z})\right|_{y=0})=\tr^{*_{\nabla^f,\, \nu f^*\chi}}(\left.\Rr_f(\mathcal{Y},\mathcal{Z})\right|_{y=0}),$$
for all $\mathcal{Y}, \mathcal{Z}\in T_f\diff_0(M)$ which means that $\Omtilde$ is $\ham(M,\omega)$-invariant.

The fact that $\Omtilde$ deforms $\Om$ is a consequence of our computation of the first term of $\Rr$ in Theorem \ref{theor:Ralpha}.
\end{proof}

\subsection{A formal moment map on $(\diff_0(M),\Omtilde)$}

Let $(X,\sigma)$ be a manifold equipped with a formal symplectic form $\sigma$. Assume there is an action $\cdot$ of a regular Lie group $G$ on $X$ preserving the symplectic form. An \emph{equivariant formal moment map} is a map 
$$\theta:\g \rightarrow C^{\infty}(X)[[\nu]],$$
for $\g$ the Lie algebra of $G$, such that for all $g\in G$, $\mathcal{Y}\in \g$ and $x\in X$
\begin{eqnarray} \label{eq:formalmoment}
\textrm{(formal moment map) } & & \imath\left(\left.\frac{d}{dt}\right|_{t=0}\exp(t\mathcal{Y})\cdot x\right)\sigma  =  d^X \theta(\mathcal{Y}) \\
\textrm{(equivariance) }& & \theta(Ad(g)\mathcal{Y}) = (g\cdot)^*\theta(\mathcal{Y}). \nonumber
\end{eqnarray}

\begin{rem}
Regular Lie group means Lie group admitting an exponential map which is not necessarily true for Fr\'echet Lie group. 
In our context, the group is the group of Hamiltonian diffeomorphisms with Lie algebra the space $C^{\infty}_0(M)$ of functions on $M$ with integral equals to $0$. The exponential of $F\in C^{\infty}_0(M)$ is given by the flow of the Hamiltonian vector field $X_F$.
\end{rem}

\begin{rem}
This formulation of formal moment map as taking values in formal functions on the manifold comes from the notion of quantum moment maps in deformation quantization \cite{gr3}.
\end{rem}

\begin{theorem} \label{theor:formalDonaldson}
The map 
$$\widetilde{\mu}: C^{\infty}_0(M) \rightarrow C^{\infty}(\diff_0(M))[[\nu]]:H \mapsto \left[ f\mapsto (2\pi)^m\nu^{m-1}\tr^{*_{\nabla^f,\,\nu f^*\chi}}(H)\right]$$
is an equivariant formal moment map for the action of $\ham(M,\omega)$ on $(\diff_0(M),\Omtilde)$.\\
Moreover, at first order in $\nu$, we recover Donaldson moment map for the action of $\ham(M,\omega)$ on $(\diff_0(M),\Om)$
\end{theorem}

\noindent We will use the next two Lemmas.

\begin{lemma} \cite{gr3} \label{lemme:Liederiv}
Consider a smooth map $t\in [0,1]\mapsto H_t\in C^{\infty}(M)$, then the derivative of the action of $\varphi_t^{H_{\cdot}}$ on $\Gamma\W\otimes\Lambda M$ is given by the formula:
\begin{equation*}
\ddt (\varphi_t^{H_{\cdot}})^*=(\varphi_t^{H_{\cdot}})^*\left(\imath(X_{H_t})D+D\imath(X_{H_t})+\invnu\left[- \omega_{ij}y^i X_{H_t}^j + \frac{1}{2}(\nabla^2_{kq}H_t) y^k y^q-\imath(X_{H_t})r,\cdot \right]\right),
\end{equation*}
where $D$ is obtained with symplectic connection $\nabla$ and the choice of a series of closed $2$-forms.
\end{lemma}

\begin{lemma}\label{lemme:QHformula}
Let $H\in C^{\infty}(M)$, $\nabla\in \E(M,\omega)$ and $f\in \diff_0(M)$, we have:
$$Q^{\nabla^f,\,\nu f^*\chi}(H)=H-\omega_{ij}y^i X_{H}^j + \frac{1}{2}((\nabla^f)^2_{kq}H) y^k y^q-\imath(X_{H})r^{\nabla^f,\,\nu f^*\chi}+ \alpha_{f}(f_* X_H).$$
\end{lemma}

The proof of the Lemma \ref{lemme:QHformula} follows by direct computations, using Lemma \ref{lemme:Liederiv} and similar to the corresponding result from \cite{LLFlast}.

\begin{proof}
The equivariance is immediate from the naturality of the Fedosov construction. The adjoint action by $\varphi\in \ham(M,\omega)$ on $C^{\infty}_0(M)$ is given by the pull-back by $\varphi^{-1}$. So that, for all $f\in \diff_0(M)$,
$$\left[\widetilde{\mu}((\varphi^{-1})^*H)\right](f)=(2\pi)^m\nu^{m-1}\tr^{*_{\nabla^f,\nu f^*\chi}}((\varphi^{-1})^*H)=(2\pi)^m\nu^{m-1}\tr^{*_{\nabla^{f\circ \varphi},\nu (f\circ \varphi)^*\chi}}(H)=\left[\widetilde{\mu}(H)\right](f\circ\varphi)$$

To check the formal moment map equation, we proceed as in \cite{LLFlast}. One compute, for $f\in\diff_0(M)$, $Y\circ f \in T_f\diff_0(M)$ and $\phi_t^Y$ the flow of $Y$ on $M$,
$$ \left(d^{\diff_0}\widetilde{\mu}(H)\right)(Y\circ f)=(2\pi)^m\nu^{m-1}\ddto\tr^{*_{\nabla^{f\circ \phi^Y_t},\,\nu (f\circ \phi^Y_t)^*\chi}}(H)$$
Using Lemma \ref{lemme:tracevariation} and after the formulas from Lemma \ref{lemme:QHformula} and \ref{lemme:Liederiv},
\begin{eqnarray*}
\ddto\tr^{*_{\nabla^{f\circ \phi^Y_t},\,\nu (f\circ \phi^Y_t)^*\chi}}(H) & = & \tr^{*_{\nabla^{f},\,\nu f^*\chi}}\left(\left.\invnu[\alpha_{f}(Y\circ f),Q^{\nabla^{f},\,\nu f^*\chi}(H)]\right|_{y=0}\right),  \\
& = & \tr^{*_{\nabla^{f},\,\nu f^*\chi}}\left(\left.\invnu[\alpha_{f}(Y\circ f),\alpha_{f}(f_* X_H)]\right|_{y=0}\right) \\
& & +\tr^{*_{\nabla^{f},\,\nu f^*\chi}}\left(\left.-\ddto (\varphi_t^H)^*\alpha_{f}(Y\circ f) \right|_{y=0}\right)\\
& & +\tr^{*_{\nabla^{f},\,\nu f^*\chi}}\left(\left.\left(\imath(X_{H_t})D^{\nabla^{f},\,\nu f^*\chi}+D^{\nabla^{f},\,\nu f^*\chi}\imath(X_{H_t})\right)\alpha_{f}(Y\circ f)\right|_{y=0}\right).
\end{eqnarray*}
Now, $\alpha_{f}(Y\circ f)$ is a $0$-form and at $y=0$ it remains,
$$\ddto\tr^{*_{\nabla^{f\circ \phi^Y_t},\,\nu (f\circ \phi^Y_t)^*\chi}}(H) =\tr^{*_{\nabla^{f},\,\nu f^*\chi}}\left(\left.\invnu[\alpha_{f}(Y\circ f),\alpha_{f}(f_* X_H)] + \imath(X_{H_t})D^{\nabla^{f},\,\nu f^*\chi}\alpha_{f}(Y\circ f)\right|_{y=0}\right).$$
By the definition of $\alpha$, we have
$$\left.\imath(X_{H_t})D^{\nabla^{f},\,\nu f^*\chi}\alpha_{f}(Y\circ f)\right|_{y=0}=-\nu f^*\chi(Y,f_*X_H).$$
Finally, by Theorem \ref{theor:Ralpha} giving the formula for $\Rr$ and Equation \eqref{eq:infXH} giving the infinitesimal action of $X_H$ on $\diff_0(M)$, we have
$$\ddto\tr^{*_{\nabla^{f\circ \phi^Y_t},\,\nu (f\circ \phi^Y_t)^*\chi}}(H) =\tr^{*_{\nabla^{f\circ \phi^Y_t},\,\nu (f\circ \phi^Y_t)^*\chi}}(\left.\Rr_f(Y\circ f, f_*X_H)\right|_{y=0}).$$
Multiplying both sides of the above equation by the constant $(2\pi)^m\nu^{m-1}$ one obtains the formal moment map equation \eqref{eq:formalmoment}.

At first order in $\nu$, one knows the first term of the normalised trace
$$(2\pi)^m\nu^{m-1}\tr^{*_{\nabla^{f},\,\nu f^*\chi}}(H)=-\int_M H (f^*\chi)\wedge \dsubvol + O(\nu),$$
which starts by the moment map from \cite{Don1}.
\end{proof}

We finish this section by the characterization of parallel transport along a path of Hamiltonian diffeomorphisms as Hamiltonian automorphisms of the star product \cite{LLF0}.

Let $H_t \in C^{\infty}(M)$ generating $\varphi_t^{H_{\cdot}}$. Consider the smooth path of connections $t\mapsto (\varphi_t^{H_{\cdot}})^*\nabla$ and $t \mapsto (\varphi_t^{H_{\cdot}})^*\chi$, for the symplectic connection $\nabla$. Recall that by naturality of the Fedosov construction $(\varphi_t^{H_{\cdot}})^{*}$ is an isomorphism of flat sections algebra:
$$(\varphi_t^{H_{\cdot}})^{*}:\Gamma\W_{D^{\nabla,\,\nu \chi}} \stackrel{\cong}{\rightarrow} \Gamma\W_{D^{(\varphi_t^{H_{\cdot}})^*\nabla, \,\nu (\varphi_t^{H_{\cdot}})^*\chi}}.$$

Now, we consider $v_t\in\Gamma\W^+$ generated by $h_t:=-\alpha_{\varphi_t^{H_{\cdot}}}(\frac{d}{dt}\varphi_t^{H_{\cdot}})$ obtained from Theorem \ref{theor:smoothisom}. The conjugation $a \in \Gamma\W_{D^{\nabla}}\mapsto v_t\circ a\circ v_t^{-1}\in \Gamma\W_{D^{(\varphi_t^{H_{\cdot}})^*\nabla, \,\nu (\varphi_t^{H_{\cdot}})^*\chi}}$  gives the parallel transport for the formal connection $\D$ along $\{\varphi_t^{H_{\cdot}}\}$.

We define the automorphism $B_t$ of $(C^{\infty}(M)[[\nu]],*_{\nabla,\, \nu \chi})$ by
$$B_t:C^{\infty}(M)[[\nu]]\rightarrow C^{\infty}(M)[[\nu]] :\left.((\varphi_t^{H_{\cdot}})^{-1})^*(v_t\circ Q^{\nabla,\,\nu\chi}(F) \circ v_t^{-1})\right|_{y=0}. $$
Then, using Lemmas \ref{lemme:Liederiv} and \ref{lemme:QHformula}, similarly as in \cite{LLFlast}, one shows:

\begin{prop} \label{prop:Btparallel}
For all $F\in C^{\infty}(M)[[\nu]]$ we have
\begin{equation*} \label{eq:Hamparallel}
\left\{\begin{array}{ccl}
\ddt B_t(F) & = &- \invnu[H_t,B_t(F)]_{*_{\nabla,\,\nu \chi}},\\
B_0(F)& = & F.
\end{array}
\right.
\end{equation*}
Hence, $B_t$ is a Hamiltonian automorphism of $*_{\nabla,\,\nu \chi}$.
\end{prop}

Everything we have done in the paper depends on the choice of a symplectic connection. We postpone the analysis of this dependence to a future work.

\end{document}